\documentclass[12pt,twoside]{amsart}
\usepackage{amsmath, amsthm, amscd, amsfonts, amssymb, graphicx}
\usepackage{enumerate}
\usepackage[colorlinks=true,
linkcolor=blue,
urlcolor=cyan,
citecolor=red]{hyperref}
\usepackage{mathrsfs}
\newtheorem{theorem}{Theorem}[section]
\newtheorem{lemma}{Lemma}[section]

\newtheorem{corollary}{Corollary}[section]

\numberwithin{equation}{section}

\begin{document}
\title{New Kantorovich type inequalities for negative parameters}
\author{Shigeru Furuichi and Hamid Reza Moradi}
\subjclass[2010]{Primary 47A63, Secondary 46L05, 47A60.}
\keywords{Kantorovich type inequality, chaotic order, order reversing operators, log-convex functions, Mond-Pe\v cari\'c method} 
 
\maketitle

\begin{abstract}
We show the following result: Let $A,B\in \mathbb{B}\left( \mathcal{H} \right)$ be two strictly positive operators such that $A\le B$ and $m{{\mathbf{1}}_{\mathcal{H}}}\le B\le M{{\mathbf{1}}_{\mathcal{H}}}$ for some scalars $0<m<M$. Then
\[{{B}^{p}}\le \exp \left( \frac{M{{\mathbf{1}}_{\mathcal{H}}}-B}{M-m}\ln {{m}^{p}}+\frac{B-m{{\mathbf{1}}_{\mathcal{H}}}}{M-m}\ln {{M}^{p}} \right)\le K\left( m,M,p,q \right){{A}^{q}}\quad\text{ for }p\le 0,-1\le q\le 0\]
where $K\left( m,M,p,q \right)$ is the generalized Kantorovich constant with two parameters. In addition, we obtain Kantorovich type inequalities for the chaotic order.
\end{abstract}
\pagestyle{myheadings}
\markboth{\centerline {New Inequalities of the Kantorovich Type With Two Negative Parameters}}
{\centerline {S. Furuichi \& H.R. Moradi}}
\bigskip
\bigskip
\section{\bf Introduction and Preliminaries}
\vskip 0.4 true cm
In what follows, a capital letter means a bounded linear operator on a complex
Hilbert space $\mathcal{H}$. An operator $A$ is said to be {\it positive} (denoted by $A\geq 0$ ) if $\left\langle Ax,x \right\rangle \ge 0$ for all $x\in \mathcal{H},$ and also an operator $A$ is said to be {\it strictly positive} (denoted by $A>0$) if $A$ is positive and invertible. Here ${{\mathbf{1}}_{\mathcal{H}}}$ stands for the identity operator on $\mathcal{H}$. $Sp\left( A \right)$ denotes the usual spectrum of $A$. We take an interval $I \subseteq \mathbb{R}$. If a positive function $f:I\to \left( 0,\infty  \right)$ satisfies
\begin{equation}\label{20}
f\left( \left( 1-v \right)x+vy \right)\le {{f^{1-v}\left( x \right) }}{{ f^{v}\left( y \right) }},	
\end{equation}
for all $x,y\in I$  and $v\in \left[ 0,1 \right]$, then we say that $f$ is a {\it logarithmically convex} (or simply, {\it log-convex}) function on $I$. The  weighted arithmetic-geometric mean inequality readily yields that every log-convex function is also convex. It is worth emphasizing that the function $f\left( t \right)={{t}^{p}}$ is log-convex for $p\le 0$ on $\left( 0,\infty  \right)$.

The \lq\lq L\"owner-Heinz inequality'' asserts that  $0\le A\le B$ ensures ${{A}^{p}}\le {{B}^{p}}$ for any $p\in \left[ 0,1 \right]$. As is well-known, the L\"owner-Heinz inequality does not always hold for $p>1$. The following theorem due to Furuta \cite[Theorem 2.1]{2} (see also \cite[Theorem 4.1]{6}) is the starting point for our discussion.

\begin{theorem}
Let $A,B\in \mathbb{B}\left( \mathcal{H} \right)$ be two strictly positive operators such that $A\le B$ and $m{{\mathbf{1}}_{\mathcal{H}}}\le A\le M{{\mathbf{1}}_{\mathcal{H}}}$  for some scalars $0<m<M$. Then 
	\[{{A}^{p}}\le {{K}}\left( m,M,p \right){{B}^{p}}\le {{\left( \frac{M}{m} \right)}^{p-1}}{{B}^{p}}\quad\text{ for }p\ge 1,\] 
where $K\left( m,M,p \right)$ is a generalized Kantorovich constant in the sense of Furuta \cite{9}:
\begin{equation}\label{9}
K\left( m,M,p \right)=\frac{(m{{M}^{p}}-M{{m}^{p}})}{\left( p-1 \right)\left( M-m \right)}{{\left( \frac{p-1}{p}\frac{{{M}^{p}}-{{m}^{p}}}{m{{M}^{p}}-M{{m}^{p}}} \right)}^{p}}\quad\text{ for }p\in \mathbb{R}.
\end{equation}
\end{theorem}

In \cite[Theorem 2.1]{3}, Mi\'ci\'c, Pe\v cari\'c and Seo proved some fascinating results about the function
preserving the operator order, under a general setting:
\begin{theorem}\label{th1.2}
 Let $A$ and $B$ be two strictly positive operators on a Hilbert space $\mathcal{H}$ satisfying $m{{\mathbf{1}}_{\mathcal{H}}}\le A\le M{{\mathbf{1}}_{\mathcal{H}}}$  for some scalars $0<m<M$. Let $f:\left[ m,M \right]\to \mathbb{R}$ be a convex function and $g:I\to \mathbb{R}$, where $I$ be any interval containing $Sp\left( B \right)\cup \left[ m,M \right]$. Suppose that either of the following conditions holds: (i) $g$ is increasing convex on $I$, or (ii) $g$ is decreasing concave on $I$. If $A\le B$, then for a given $\alpha >0$ in the case (i) or $\alpha <0$ in the case (ii)
\[f\left( A \right)\le \alpha g\left( B \right)+\beta {{\mathbf{1}}_{\mathcal{H}}},\]
holds for
\begin{equation}\label{12}
\beta =\underset{m\le t\le M}{\mathop{\max }}\,\left\{ {{a}_{f}}t+{{b}_{f}}-\alpha g\left( t \right) \right\},
\end{equation}
where
\[{{a}_{f}}:= \frac{f\left( M \right)-f\left( m \right)}{M-m}\quad\text{ and }\quad{{b}_{f}}:= \frac{Mf\left( m \right)-mf\left( M \right)}{M-m}.\]
\end{theorem}

The following converse of Theorem \ref{th1.2} was proven in \cite[Theorem 2.1]{4}:
\begin{theorem}\label{c}
Let $A$ and $B$ be two strictly positive operators on a Hilbert space $\mathcal{H}$ satisfying $m{{\mathbf{1}}_{\mathcal{H}}}\le B\le M{{\mathbf{1}}_{\mathcal{H}}}$  for some scalars $0<m<M$. Let $f:\left[ m,M \right]\to \mathbb{R}$ be a convex function and $g:I\to \mathbb{R}$, where $I$ be any interval containing $Sp\left( A \right)\cup \left[ m,M \right]$. Suppose that either of the following conditions holds: (i) $g$ is decreasing convex on $I$, or (ii) $g$ is increasing concave on $I$. If $A\le B$, then for a given $\alpha >0$ in the case (i) or $\alpha <0$ in the case (ii)
\begin{equation}\label{11}
f\left( B \right)\le \alpha g\left( A \right)+\beta {{\mathbf{1}}_{\mathcal{H}}},
\end{equation}
holds with $\beta$ as \eqref{12}.
\end{theorem}

\medskip

This paper is organized by four sections. The proof of our main
result, Theorem \ref{a}, is given in Section \ref{s1}. The essential idea is to consider the log-convex function instead of the convex function in Theorem \ref{c}.  As applications, in Section \ref{s3}, we show some characterizations of the chaotic order. Further results based on the Mond-Pe\v cari\'c method are given in Section \ref{s4}.
\section{\bf Functions reversing operator order}\label{s1}
\vskip 0.4 true cm
In the sequel, ${{a}_{f}}$ and ${{b}_{f}}$ will refer to those of Theorem \ref{th1.2}.
Our principal result is the following theorem. The role of \eqref{20} is clearly brought out in our proof.
\begin{theorem}\label{a}
Let $A,B\in \mathbb{B}\left( \mathcal{H} \right)$ be two self-adjoint operators such that $m{{\mathbf{1}}_{\mathcal{H}}}\le B\le M{{\mathbf{1}}_{\mathcal{H}}}$ for some scalars $m<M$. Let $f:\left[ m,M \right]\to \left( 0,\infty  \right)$ be a log-convex function and $g:I\to \mathbb{R}$, where $I$ be any interval containing $Sp\left( A \right)\cup \left[ m,M \right]$. Suppose that either of the
following conditions holds: (i) $g$ is decreasing convex on $I$, or (ii) $g$ is increasing concave on $I$.  If $A\le B$, then for a given $\alpha >0$ in the case (i) or $\alpha <0$ in the case (ii)
\begin{equation}\label{ineq00_theorem_2_1}
f\left( B \right)\le \exp \left( \frac{M{{\mathbf{1}}_{\mathcal{H}}}-B}{M-m}\ln f\left( m \right)+\frac{B-m{{\mathbf{1}}_{\mathcal{H}}}}{M-m}\ln f\left( M \right) \right)\le \alpha g\left( A \right)+\beta {{\mathbf{1}}_{\mathcal{H}}}, 
\end{equation}
holds with $\beta$ as in \eqref{12}.
\end{theorem}
\begin{proof}
We prove the inequalities (\ref{ineq00_theorem_2_1}) under the assumption (i). It is immediate to see that
\begin{equation}\label{5}
f\left( t \right)\le {{ f^{\frac{M-t}{M-m}}\left( m \right) }}{{ f^{\frac{t-m}{M-m}}\left( M \right) }}\le L\left( t \right)\quad\text{ for }m\le t\le M,
\end{equation}
where
\[L\left( t \right)=\frac{M-t}{M-m}f\left( m \right)+\frac{t-m}{M-m}f\left( M \right)={{a}_{f}}t+{{b}_{f}}.\]
By applying the standard operational calculus of self-adjoint operator $B$ to \eqref{5}, we obtain for every unit vector $x\in \mathcal{H}$,
\[\left\langle f\left( B \right)x,x \right\rangle \le \left\langle \exp \left( \frac{M{{\mathbf{1}}_{\mathcal{H}}}-B}{M-m}\ln f\left( m \right)+\frac{B-m{{\mathbf{1}}_{\mathcal{H}}}}{M-m}\ln f\left( M \right) \right)x,x \right\rangle \le {{a}_{f}}\left\langle Bx,x \right\rangle +{{b}_{f}},\]
and from this it follows that
\[\begin{aligned}
& \left\langle f\left( B \right)x,x \right\rangle -\alpha g\left( \left\langle Bx,x \right\rangle  \right) \\ 
&\quad \le \left\langle \exp \left( \frac{M{{\mathbf{1}}_{\mathcal{H}}}-B}{M-m}\ln f\left( m \right)+\frac{B-m{{\mathbf{1}}_{\mathcal{H}}}}{M-m}\ln f\left( M \right) \right)x,x \right\rangle -\alpha g\left( \left\langle Bx,x \right\rangle  \right) \\ 
&\quad \le {{a}_{f}}\left\langle Bx,x \right\rangle +{{b}_{f}}-\alpha g\left( \left\langle Bx,x \right\rangle  \right) \\ 
&\quad \le \underset{m\le t\le M}{\mathop{\max }}\,\left\{ {{a}_{f}}t+{{b}_{f}}-\alpha g\left( t \right) \right\}. \\ 
\end{aligned}\]
Here we put $t:=\langle Bx,x\rangle$, then $m \leq t \leq M$.
Whence
\[\begin{aligned}
 \left\langle f\left( B \right)x,x \right\rangle &\le \left\langle \exp \left( \frac{M{{\mathbf{1}}_{\mathcal{H}}}-B}{M-m}\ln f\left( m \right)+\frac{B-m{{\mathbf{1}}_{\mathcal{H}}}}{M-m}\ln f\left( M \right) \right)x,x \right\rangle  \\ 
& \le \alpha g\left( \left\langle Bx,x \right\rangle  \right)+\beta  \\ 
& \le \alpha g\left( \left\langle Ax,x \right\rangle  \right)+\beta  \quad \text{(since $A\le B$ and $g$ is decreasing)}\\ 
& \le \alpha \left\langle g\left( A \right)x,x \right\rangle +\beta \quad \text{(since $g$ is convex)}  
\end{aligned}\]
and the assertion follows.
\end{proof}

\medskip

The following corollary improves the result in \cite[Corollary 2.5]{4}. In fact, if we put $f\left( t \right)={{t}^{p}}$ and $g\left( t \right)={{t}^{q}}$ with $p\leq 0$ and $q\leq 0$, we get:
\begin{corollary}\label{8}
	Let $A,B\in \mathbb{B}\left( \mathcal{H} \right)$ be two strictly positive operators such that $A\le B$ and $m{{\mathbf{1}}_{\mathcal{H}}}\le B\le M{{\mathbf{1}}_{\mathcal{H}}}$ for some scalars $0<m<M$. Then for a given $\alpha >0$,
\begin{equation}\label{18}
{{B}^{p}}\le \exp \left( \frac{M{{\mathbf{1}}_{\mathcal{H}}}-B}{M-m}\ln {{m}^{p}}+\frac{B-m{{\mathbf{1}}_{\mathcal{H}}}}{M-m}\ln {{M}^{p}} \right)\le \alpha {{A}^{q}}+\beta {{\mathbf{1}}_{\mathcal{H}}},\quad (p\leq 0,\,\, q \leq 0)
\end{equation}
holds, where $\beta $ is defined as
\begin{equation}\label{17}
\beta =\left\{ \begin{array}{ll}
\alpha \left( q-1 \right){{\left( \frac{{{M}^{p}}-{{m}^{p}}}{\alpha q\left( M-m \right)} \right)}^{\frac{q}{q-1}}}+\frac{M{{m}^{p}}-m{{M}^{p}}}{M-m}&\text{ if }m\leq {{\left( \frac{{{M}^{p}}-{{m}^{p}}}{\alpha q\left( M-m \right)} \right)}^{\frac{1}{q-1}}}\leq M \\ 
\max \left\{ {{m}^{p}}-\alpha {{m}^{q}},{{M}^{p}}-\alpha {{M}^{q}} \right\}&\text{ otherwise} \\ 
\end{array} \right..
\end{equation}
Especially, by setting $p=q$ in \eqref{18}, we reach
\[{{B}^{p}}\le \exp \left( \frac{M{{\mathbf{1}}_{\mathcal{H}}}-B}{M-m}\ln {{m}^{p}}+\frac{B-m{{\mathbf{1}}_{\mathcal{H}}}}{M-m}\ln {{M}^{p}} \right)\le \alpha {{A}^{p}}+\beta {{\mathbf{1}}_{\mathcal{H}}}\quad (p\leq 0),\]
where
\begin{equation}\label{be}
\beta =\left\{ \begin{array}{ll}
 \alpha \left( p-1 \right){{\left( \frac{{{M}^{p}}-{{m}^{p}}}{\alpha p\left( M-m \right)} \right)}^{\frac{p}{p-1}}}+\frac{M{{m}^{p}}-m{{M}^{p}}}{M-m}&\text{ if }m\leq {{\left( \frac{{{M}^{p}}-{{m}^{p}}}{\alpha p\left( M-m \right)} \right)}^{\frac{1}{p-1}}}\leq M \\ 
 \max \left\{ {{m}^{p}}-\alpha {{m}^{p}},{{M}^{p}}-\alpha {{M}^{p}} \right\}&\text{ otherwise} \\ 
\end{array} \right..
\end{equation}	
\end{corollary}

\medskip

If we choose $\alpha $ such that $\beta =0$ in Theorem \ref{a}, then we obtain the
following corollary. For completeness, we sketch the proof.
\begin{corollary}\label{cor2.3}
Let $A,B\in \mathbb{B}\left( \mathcal{H} \right)$ be two strictly positive operators such that $A\le B$ and $m{{\mathbf{1}}_{\mathcal{H}}}\le B\le M{{\mathbf{1}}_{\mathcal{H}}}$ for some scalars $0<m<M$. Let $f:\left[ m,M \right]\to \left( 0,\infty  \right)$ be a log-convex function and $g:I\to \mathbb{R}$ be a continuous function, where $I$ is an interval containing $Sp\left( A \right)\cup \left[ m,M \right]$.
 If $g$ is a positive decreasing convex function on $\left[ m,M \right]$, then
	\begin{equation}\label{7}
f\left( B \right)\le \exp \left( \frac{M{{\mathbf{1}}_{\mathcal{H}}}-B}{M-m}\ln f\left( m \right)+\frac{B-m{{\mathbf{1}}_{\mathcal{H}}}}{M-m}\ln f\left( M \right) \right)\le \underset{m\le t\le M}{\mathop{\max }}\,\left\{ \frac{{{a}_{f}}t+{{b}_{f}}}{g\left( t \right)} \right\}g\left( A \right).
	\end{equation}
Moreover if $p \leq 0$ and $-1 \leq q \leq 0$, then
\begin{equation}\label{15}
{{B}^{p}}\le \exp \left( \frac{M{{\mathbf{1}}_{\mathcal{H}}}-B}{M-m}\ln {{m}^{p}}+\frac{B-m{{\mathbf{1}}_{\mathcal{H}}}}{M-m}\ln {{M}^{p}} \right)\le {{K}}\,\left( m,M,p,q \right){{A}^{q}},
\end{equation}
where ${K}\left( m,M,p,q \right)$ (see, e.g., \cite[Theorem 3.1]{3}) is defined as 
	\begin{equation}\label{16}
{K}\left( m,M,p,q \right)=\left\{ \begin{array}{ll}
\frac{(m{{M}^{p}}-M{{m}^{p}})}{\left( q-1 \right)\left( M-m \right)}{{\left( \frac{q-1}{q}\frac{{{M}^{p}}-{{m}^{p}}}{m{{M}^{p}}-M{{m}^{p}}} \right)}^{q}}&\text{ if }m\leq \frac{q\left( m{{M}^{p}}-M{{m}^{p}} \right)}{\left( q-1 \right)\left( {{M}^{p}}-{{m}^{p}} \right)}\leq M \\ 
\max \left\{ {{m}^{p-q}},{{M}^{p-q}} \right\}&\text{ otherwise} \\ 
\end{array} \right..
\end{equation}

In particular, if $p=q$ in \eqref{15}, we have
\begin{equation}\label{2_8}
{{B}^{p}}\le \exp \left( \frac{M{{\mathbf{1}}_{\mathcal{H}}}-B}{M-m}\ln {{m}^{p}}+\frac{B-m{{\mathbf{1}}_{\mathcal{H}}}}{M-m}\ln {{M}^{p}} \right)\le {{K}}\left( m,M,p \right){{A}^{p}}\quad\text{ for }p\leq 0,
\end{equation}
where $K\left( m,M,p \right)$ is defined as \eqref{9}.
\end{corollary}
\begin{proof}
From the condition on the function $g$, 
we have 
$\beta \leq \underset{m\le t\le M}{\mathop{\max }}\,\left\{ {{a}_{f}}t+{{b}_{f}} \right\}-\alpha \underset{m\le t\le M}{\mathop{\min }}\,\left\{ g\left( t \right) \right\}$. When $\beta =0$, we have
$\alpha \leq \frac{\underset{m\le t\le M}{\mathop{\max }}\,\left\{ {{a}_{f}}t+{{b}_{f}} \right\}}{\underset{m\le t\le M}{\mathop{\min }}\,\left\{ g\left( t \right) \right\}}$. Thus we have the inequalities (\ref{7}) taking $\alpha :=\underset{m\le t\le M}{\mathop{\max }}\,\left\{ \frac{{{a}_{f}}t+{{b}_{f}}}{g\left( t \right)} \right\}$ which is less than or equal to  $\frac{\underset{m\le t\le M}{\mathop{\max }}\,\left\{ {{a}_{f}}t+{{b}_{f}} \right\}}{\underset{m\le t\le M}{\mathop{\min }}\,\left\{ g\left( t \right) \right\}}$.

If we take $f\left( t \right)={{t}^{p}}$ and $g\left( t \right)={{t}^{q}}$ with $p\leq 0$ and $-1 \leq q \leq 0$ for $t>0$ in \eqref{7}, then we have $a_{t^p} \leq 0$, $t_{b^{p}} \geq 0$ and $\alpha =\underset{m\le t\le M}{\mathop{\max }}\,\left\{ {{a}_{{{t}^{p}}}}{{t}^{1-q}}+{{b}_{{{t}^{p}}}}{{t}^{-q}} \right\}$. Then we set $h_{p,q}(t):=  {{a}_{{{t}^{p}}}}{{t}^{1-q}}+{{b}_{{{t}^{p}}}}{{t}^{-q}}$. We easily calculate 
$$
h_{p,q}'(t)=t^{-q-1}\left\{(1-q)a_{t^{p}}t-qb_{t^{p}}\right\},\quad
h_{p,q}''(t) = t^{-q-2} \left\{q(q-1)a_{t^{p}} t +q(q+1) b_{t^{p}}\right\}\leq 0.
$$ 
We find $\alpha= \left(\frac{b_{t^p}}{1-q}\right) \left\{ \frac{(1-q) a_{t^p}}{q b_{t^p}}\right\}^q$ if  $t_0 := \frac{q b_{t^p}}{(1-q)a_{t^p}}$ satisfies $m \leq t_0 \leq M$. 

Thus we have $\alpha =K(m,M,p,q)$ by simple calculations with ${{a}_{{{t}^{p}}}}=\frac{{{M}^{p}}-{{m}^{p}}}{M-m}$, ${{b}_{{{t}^{p}}}}=\frac{M{{m}^{p}}-m{{M}^{p}}}{M-m}$ and the other cases are trivial. Thus we have the inequalities (\ref{15}) and (\ref{2_8}).
\end{proof}
Observe that Corollary \ref{cor2.3} gives a refinement of \cite[Corollary 2.6]{4}.
In addition, for example we take $p=q=-1$, then $\alpha = K(m,M,-1,-1) = \frac{(M+m)^2}{4Mm}$ which is the original Kantorovich constant. Then we also have $\beta = \frac{M+m}{Mm}-2\sqrt{\frac{\alpha}{Mm}}$ from \eqref{be}. Inserting $\alpha =\frac{(M+m)^2}{4Mm}$ to the above, we can confirm $\beta = 0$ easily.

\medskip

The last result in this section, which is a refinement of \cite[Corollary 2.2]{kim} (see also \cite[Corollary 1]{5}) can be stated as follows.
\begin{corollary}\label{theorem_2_3}
	Let $A,B\in \mathbb{B}\left( \mathcal{H} \right)$ be two strictly positive operators such that $A\le B$ and $m{{\mathbf{1}}_{\mathcal{H}}}\le B\le M{{\mathbf{1}}_{\mathcal{H}}}$ for some scalars $0<m<M$. Then
	\begin{equation}\label{10}
	{{B}^{p}}\le { \exp \left( \frac{M{{\mathbf{1}}_{\mathcal{H}}}-B}{M-m}\ln m^p+\frac{B-m{{\mathbf{1}}_{\mathcal{H}}}}{M-m}\ln M^p \right) }\le C\left( m,M,p,q \right){{\mathbf{1}}_{\mathcal{H}}}+{{A}^{q}}\quad\text{ for }p,q\leq 0,
	\end{equation}
	where $C\left( m,M,p,q \right)$ is the Kantorovich constant for the difference with two parameters and defined by
	\[C\left( m,M,p,q \right)=\left\{ \begin{array}{ll}
	\frac{M{{m}^{p}}-m{{M}^{p}}}{M-m}+\left( q-1 \right){{\left( \frac{{{M}^{p}}-{{m}^{p}}}{q\left( M-m \right)} \right)}^{\frac{q}{q-1}}}&\text{ if }m\le {{\left( \frac{{{M}^{p}}-{{m}^{p}}}{q\left( M-m \right)} \right)}^{\frac{1}{q-1}}}\le M \\ 
	\max \left\{ {{M}^{p}}-{{M}^{q}},{{m}^{p}}-{{m}^{q}} \right\}&\text{ otherwise} \\ 
	\end{array} \right..\]
\end{corollary}
\begin{proof}
If we put $\alpha =1$, $f\left( t \right)={{t}^{p}}$ for $p \leq 0$ and $g(t) =t^q$ for $q \leq 0$ in Theorem \ref{a}, then we have
$\beta =\underset{m\le t\le M}{\mathop{\max }}\,\left\{ {{a}_{{{t}^{p}}}}t+{{b}_{{{t}^{p}}}}-{{t}^{q}} \right\}$. By simple calculations, we have $\beta = (q-1)\left(\frac{a_{t^p}}{q}\right)^{\frac{q}{q-1}} +b_{t^p}$ if $t_0 := \left( \frac{a_{t^p}}{q}\right)^{\frac{1}{q-1}}$  satisfies $m \leq t_0 \leq M$. The other cases are trivial. Thus we have the desired conclusion, since ${{a}_{{{t}^{p}}}}=\frac{{{M}^{p}}-{{m}^{p}}}{M-m}$ and ${{b}_{{{t}^{p}}}}=\frac{M{{m}^{p}}-m{{M}^{p}}}{M-m}$.
\end{proof}


\section{\bf Applications to chaotic order}\label{s3}
\vskip 0.4 true cm
In this section, we show some inequalities on chaotic order (i.e., $\log A\le \log B$ for $A,B>0$). To achieve our next results, we need the following lemma. Its proof is standard but we provide a proof for the sake of completeness.
\begin{lemma} \label{lemma2_1}
Let $A,B\in \mathbb{B}\left( \mathcal{H} \right)$ be two strictly positive operators.
Then the following statements are equivalent:
\begin{itemize}
\item[(i)] $\log A \leq \log B$.
\item[(ii)] $B^{r} \leq \left( B^{\frac{r}{2}} A^p B^{\frac{r}{2}}\right)^{\frac{r}{p+r}}$\quad {\rm for} $p\leq 0$  {\rm and}  $r\le 0$.
\end{itemize}
\end{lemma}

\begin{proof}
From the well-known \lq\lq chaotic Furuta inequality'' (see, e.g., \cite{FFK1993,FJK1997,F1992}) the order $\log A \geq \log B$ is equivalent to the inequality $\left(B^{\frac{r}{2}}A^pB^{\frac{r}{2}}\right)^{\frac{p}{p+r}} \geq B^r$ for $p,r\ge 0$ and $A,B >0$.  The assertion (i)  is equivalent to the order  $\log B^{-1} \leq \log A^{-1}$. By the use of chaotic Furuta inequality,  the order  $\log B^{-1} \leq \log A^{-1}$  is equivalent to the inequality 
\begin{equation}  \label{proof_lemma2_1_ineq01}
{{B}^{-r}}\le {{\left( {{B}^{\frac{-r}{2}}}{{A}^{-p}}{{B}^{\frac{-r}{2}}} \right)}^{\frac{r}{p+r}}}\quad\text{ for }p,r\ge 0.
\end{equation}
This is  equivalent to the inequality 
$$
{{B}^{r'}}\le {{\left( {{B}^{\frac{r'}{2}}}{{A}^{p'}}{{B}^{\frac{r'}{2}}} \right)}^{\frac{r'}{p'+r'}}}\quad\text{ for }p',r'\le 0,$$
by substituting $p' = -p$ and $r'=-r$ in (\ref{proof_lemma2_1_ineq01}).
We thus obtain the desired conclusion.
\end{proof}

\medskip

As an application of Corollary \ref{cor2.3}, we have the following result:
\begin{corollary}
Let $A,B\in \mathbb{B}\left( \mathcal{H} \right)$ be two strictly positive  operators such that $m{{\mathbf{1}}_{\mathcal{H}}}\le B\le M{{\mathbf{1}}_{\mathcal{H}}}$ for some scalars $0<m<M$ and $\log A \leq \log B$. Then for $p \leq 0$ and $r \leq -1$, 
\[{{B}^{p}}\le B^{-r} \exp\left( \frac{M{{\mathbf{1}}_{\mathcal{H}}}-B}{M-m}\ln {{m}^{p+r}}+\frac{B-m{{\mathbf{1}}_{\mathcal{H}}}}{M-m}\ln {{M}^{p+r}} \right) \leq {{K}}\left( m,M,p+r \right){{A}^{p}}.\]
\end{corollary}
\begin{proof}
The idea of proof is similar to the one in \cite[Theorem 1]{YY1999}. Thanks to Lemma \ref{lemma2_1}, the chaotic order $\log A \leq \log B$ is equivalent to
 $B^{r} \leq \left( B^{\frac{r}{2}} A^p B^{\frac{r}{2}}\right)^{\frac{r}{p+r}}$ for $p,r\le 0$.
Putting $B_1 = B$ and $A_1=\left( B^{\frac{r}{2}}A^{p}B^{\frac{r}{2}}\right)^{\frac{1}{p+r}}$ in the above, then $0 < A_1 \leq B_1$ and $m{{\mathbf{1}}_{\mathcal{H}}}\le {{B}_{1}}\le M{{\mathbf{1}}_{\mathcal{H}}}$, since $0\le -\dfrac{1}{r} \le 1$ from the assumption $r \le -1$.
Thus we have for $p_1 \leq 0$
\[\begin{aligned}
{{B}^{{{p}_{1}}}}&=B_{1}^{{{p}_{1}}} \\ 
& \le \exp \left( \frac{M{{\mathbf{1}}_{\mathcal{H}}}-B_1}{M-m}\ln {{m}^{{{p}_{1}}}}+\frac{B_1-m{{\mathbf{1}}_{\mathcal{H}}}}{M-m}\ln {{M}^{{{p}_{1}}}} \right) \\ 
& \le {{K}}\left( m,M,{{p}_{1}} \right)A_{1}^{{{p}_{1}}} \\ 
& ={{K}}\left( m,M,{{p}_{1}} \right){{\left( {{B}^{\frac{r}{2}}}{{A}^{p}}{{B}^{\frac{r}{2}}} \right)}^{\frac{{{p}_{1}}}{p+r}}},  
\end{aligned}\]
by (\ref{2_8}). By setting $p_1 = p+r \leq 0$ and multiplying $B^{-\frac{r}{2}}$ to both sides, we obtain the desired conclusion.
\end{proof}

\medskip

In a similar fashion, one can prove the following result:
\begin{corollary}
Let $A,B\in \mathbb{B}\left( \mathcal{H} \right)$ be two strictly positive  operators such that $m{{\mathbf{1}}_{\mathcal{H}}}\le B\le M{{\mathbf{1}}_{\mathcal{H}}}$ for some scalars $0<m<M$ and $\log A \leq \log B$. Then for $p \leq 0$ and $r \leq -1$, 
\[{{B}^{p}}\le {{B}^{-r}}\exp \left( \frac{M{{\mathbf{1}}_{\mathcal{H}}}-B}{M-m}\ln {{m}^{p+r}}+\frac{B-m{{\mathbf{1}}_{\mathcal{H}}}}{M-m}\ln {{M}^{p+r}} \right)\le C\left( m,M,p+r \right)  B^{-r} 
+{{A}^{p}},\]
where
\begin{equation}\label{cm}
C\left( m,M,p \right)=\left\{ \begin{array}{ll}
\frac{M{{m}^{p}}-m{{M}^{p}}}{M-m}+\left( p-1 \right){{\left( \frac{{{M}^{p}}-{{m}^{p}}}{p\left( M-m \right)} \right)}^{\frac{p}{p-1}}}&\text{ if }m\le {{\left( \frac{{{M}^{p}}-{{m}^{p}}}{p\left( M-m \right)} \right)}^{\frac{1}{p-1}}}\le M \\ 
0&\text{ otherwise} \\ 
\end{array} \right..
\end{equation}
\end{corollary}
\begin{proof}
If we set $p=q$ in Corollary \ref{theorem_2_3}, we have the following inequalities for $p \leq 0$
\begin{equation} \label{ineq01_proof_corollary_2_5}
{{B}^{p}}\le \exp \left( \frac{M{{\mathbf{1}}_{\mathcal{H}}}-B}{M-m}\ln {{m}^{p}}+\frac{B-m{{\mathbf{1}}_{\mathcal{H}}}}{M-m}\ln {{M}^{p}} \right)\le C\left( m,M,p \right){{\mathbf{1}}_{\mathcal{H}}}+{{A}^{p}}.
\end{equation}

Thanks to Lemma \ref{lemma2_1}, the chaotic order $\log A \leq \log B$ is equivalent to
 $B^{r} \leq \left( B^{\frac{r}{2}} A^p B^{\frac{r}{2}}\right)^{\frac{r}{p+r}}$ for $p,r\le 0$.
Putting $B_1 = B$ and $A_1=\left( B^{\frac{r}{2}}A^{p}B^{\frac{r}{2}}\right)^{\frac{1}{p+r}}$ in the above, then $0 < A_1 \leq B_1$ and $m{{\mathbf{1}}_{\mathcal{H}}}\le {{B}_{1}}\le M{{\mathbf{1}}_{\mathcal{H}}}$. Thus we have for $p_1 \leq 0$
\[B_{1}^{{{p}_{1}}}\le \exp \left( \frac{M{{\mathbf{1}}_{\mathcal{H}}}-{{B}_{1}}}{M-m}\ln {{m}^{{{p}_{1}}}}+\frac{{{B}_{1}}-m{{\mathbf{1}}_{\mathcal{H}}}}{M-m}\ln {{M}^{{{p}_{1}}}} \right)\le C\left( m,M,{{p}_{1}} \right){{\mathbf{1}}_{\mathcal{H}}}+A_{1}^{{{p}_{1}}},\]
by (\ref{ineq01_proof_corollary_2_5}). Putting $p_1 = p+r \leq 0$ and multiplying $B^{-\frac{r}{2}}$ to both sides, we obtain the desired conclusion.
\end{proof}

\section{\bf Miscellanea}\label{s4}
\vskip 0.4 true cm
By the similar way presented in the previous sections, it is also possible to improve the results which previously obtained by employing the Mond-Pe\v cari\'c method. 

As a multiple operator version of the celebrated \lq\lq Davis-Choi-Jensen inequality'' \cite{choi}, Mond and Pe\v cari\'c in \cite[Theorem 1]{mond} proved the inequality
\begin{equation}\label{21}
f\left( \sum\limits_{i=1}^{n}{{{w}_{i}}{{\Phi }_{i}}\left( {{A}_{i}} \right)} \right)\le \sum\limits_{i=1}^{n}{{{w}_{i}}{{\Phi }_{i}}\left( f\left( {{A}_{i}} \right) \right)},
\end{equation}
for operator convex function $f$ defined on an interval $I$, where ${{\Phi }_{i}}$ ($i=1,\ldots ,n$) are normalized positive linear mappings from $\mathbb{B}\left( \mathcal{H} \right)$ to $\mathbb{B}\left( \mathcal{K} \right)$, ${{A}_{1}},\ldots ,{{A}_{n}}$ are self-adjoint operators with spectra in $I$ and ${{w}_{1}},\ldots ,{{w}_{n}}$ are non-negative real numbers with $\sum\nolimits_{i=1}^{n}{{{w}_{i}}}=1$.

\medskip

In a reverse direction to that of inequality \eqref{21} we have the following:
\begin{theorem}\label{4.1}
	Let ${{\Phi }_{i}}$ be normalized positive linear maps from $\mathbb{B}\left( \mathcal{H} \right)$ to $\mathbb{B}\left( \mathcal{K} \right)$, ${{A}_{i}}\in \mathbb{B}\left( \mathcal{H} \right)$ be self-adjoint operators with $m{{\mathbf{1}}_{\mathcal{H}}}\le {{A}_{i}}\le M{{\mathbf{1}}_{\mathcal{H}}}$ for some scalars $m<M$ and ${{w}_{i}}$ be positive numbers such that $\sum\nolimits_{i=1}^{n}{{{w}_{i}}}=1$. If $f$ is a log-convex function and $g$ is a continuous function on $\left[ m,M \right]$, then for a given $\alpha \in \mathbb{R}$   
	\begin{equation}\label{19}
	\begin{aligned}
	\sum\limits_{i=1}^{n}{{{w}_{i}}{{\Phi }_{i}}\left( f\left( {{A}_{i}} \right) \right)}&\le \sum\limits_{i=1}^{n}{{{w}_{i}}{{\Phi }_{i}}\left( \exp \left( \frac{M{{\mathbf{1}}_{\mathcal{H}}}-{{A}_{i}}}{M-m}\ln f\left( m \right)+\frac{{{A}_{i}}-m{{\mathbf{1}}_{\mathcal{H}}}}{M-m}\ln f\left( M \right) \right) \right)} \\ 
	& \le \alpha g\left( \sum\limits_{i=1}^{n}{{{w}_{i}}{{\Phi }_{i}}\left( {{A}_{i}} \right)} \right)+\beta {{\mathbf{1}}_{\mathcal{K}}}, 
	\end{aligned}
	\end{equation}   
	holds with $\beta$ as in \eqref{12}.
\end{theorem}
\begin{proof}
	Thanks to \eqref{5}, we get
	\[f\left( {{A}_{i}} \right)\le \exp \left( \frac{M{{\mathbf{1}}_{\mathcal{H}}}-{{A}_{i}}}{M-m}\ln f\left( m \right)+\frac{{{A}_{i}}-m{{\mathbf{1}}_{\mathcal{H}}}}{M-m}\ln f\left( M \right) \right)\le {{a}_{f}}{{A}_{i}}+{{b}_{f}}{{\mathbf{1}}_{\mathcal{H}}}.\]
	The hypotheses on ${{\Phi }_{i}}$ and ${{w}_{i}}$ ensure the following: 
	\[\begin{aligned}
	\sum\limits_{i=1}^{n}{{{w}_{i}}{{\Phi }_{i}}\left( f\left( {{A}_{i}} \right) \right)}&\le \sum\limits_{i=1}^{n}{{{w}_{i}}{{\Phi }_{i}}\left( \exp \left( \frac{M{{\mathbf{1}}_{\mathcal{H}}}-{{A}_{i}}}{M-m}\ln f\left( m \right)+\frac{{{A}_{i}}-m{{\mathbf{1}}_{\mathcal{H}}}}{M-m}\ln f\left( M \right) \right) \right)} \\ 
	& \le {{a}_{f}}\sum\limits_{i=1}^{n}{{{w}_{i}}{{\Phi }_{i}}\left( {{A}_{i}} \right)}+{{b}_{f}}{{\mathbf{1}}_{\mathcal{K}}}.  
	\end{aligned}\]
	Using the fact that $m{{\mathbf{1}}_{\mathcal{K}}}\le \sum\limits_{i=1}^{n}{{{w}_{i}}{{\Phi }_{i}}\left( {{A}_{i}} \right)}\le M{{\mathbf{1}}_{\mathcal{K}}}$, we can write
	\[\begin{aligned}
	& \sum\limits_{i=1}^{n}{{{w}_{i}}{{\Phi }_{i}}\left( f\left( {{A}_{i}} \right) \right)-\alpha g\left( \sum\limits_{i=1}^{n}{{{w}_{i}}{{\Phi }_{i}}\left( {{A}_{i}} \right)} \right)} \\ 
	& \quad\le \sum\limits_{i=1}^{n}{{{w}_{i}}{{\Phi }_{i}}\left( \exp \left( \frac{M{{\mathbf{1}}_{\mathcal{H}}}-{{A}_{i}}}{M-m}\ln f\left( m \right)+\frac{{{A}_{i}}-m{{\mathbf{1}}_{\mathcal{H}}}}{M-m}\ln f\left( M \right) \right) \right)}-\alpha g\left( \sum\limits_{i=1}^{n}{{{w}_{i}}{{\Phi }_{i}}\left( {{A}_{i}} \right)} \right) \\ 
	&\quad \le {{a}_{f}}\sum\limits_{i=1}^{n}{{{w}_{i}}{{\Phi }_{i}}\left( {{A}_{i}} \right)}+{{b}_{f}}{{\mathbf{1}}_{\mathcal{K}}}-\alpha g\left( \sum\limits_{i=1}^{n}{{{w}_{i}}{{\Phi }_{i}}\left( {{A}_{i}} \right)} \right) \\ 
	&\quad \le \underset{m\le t\le M}{\mathop{\max }}\,\left\{ {{a}_{f}}t+{{b}_{f}}-\alpha g\left( t \right) \right\}{{\mathbf{1}}_{\mathcal{K}}}, 
	\end{aligned}\]
	which is, after rearrangement, equivalent to \eqref{19}. So the proof is complete. 
\end{proof}
It is worth mentioning that, Theorem \ref{4.1} is stronger than what appears in \cite[Theorem 2.2]{05}.

\medskip

Let $I\subseteq \mathbb{R}$ and $f:I\to \mathbb{R}$ be a continuous function and let $A,B\in \mathbb{B}\left( \mathcal{H} \right)$ be two strictly positive  operators such that $Sp\left( {{A}^{-\frac{1}{2}}}B{{A}^{-\frac{1}{2}}} \right)\subseteq I$. Then the operator ${{\sigma }_{f}}$ given by
\[A{{\sigma }_{f}}B={{A}^{\frac{1}{2}}}f\left( {{A}^{-\frac{1}{2}}}B{{A}^{-\frac{1}{2}}} \right){{A}^{\frac{1}{2}}},\] 
is called {\it $f$-connection} (cf. \cite{kubo}). We shall show the following result involving $f$-connection of strictly positive operators.
\begin{theorem}\label{04.1}
	Let $\Phi $ be a normalized positive linear map from $\mathbb{B}\left( \mathcal{H} \right)$ to $\mathbb{B}\left( \mathcal{K} \right)$ and $A,B\in \mathbb{B}\left( \mathcal{H} \right)$ be two strictly positive operators such that $mA\le B\le MA$ for some scalars $0<m<M$. If $f$ is a log-convex function on $\left[ m,M \right]$, then for a given $\alpha \in \mathbb{R}$   
	\begin{equation}\label{02}
\begin{aligned}
\Phi \left( A{{\sigma }_{f}}B \right)&\le \Phi \left( {{A}^{\frac{1}{2}}}\exp \left( \frac{M{{\mathbf{1}}_{\mathcal{H}}}-{{A}^{-\frac{1}{2}}}B{{A}^{-\frac{1}{2}}}}{M-m}\ln f\left( m \right)+\frac{{{A}^{-\frac{1}{2}}}B{{A}^{-\frac{1}{2}}}-m{{\mathbf{1}}_{\mathcal{H}}}}{M-m}\ln f\left( M \right) \right){{A}^{\frac{1}{2}}} \right) \\ 
& \le \beta \Phi \left( A \right)+\alpha \left( \Phi \left( A \right){{\sigma }_{f}}\Phi \left( B \right) \right),  
\end{aligned}
\end{equation}
holds with $\beta$ as in \eqref{12}.
\end{theorem}
\begin{proof}
We give a lengthy sketch but routine calculations.	It follows from Theorem \ref{4.1} that
\begin{equation}\label{01}
\begin{aligned}
\Psi \left( f\left( {{A}^{-\frac{1}{2}}}B{{A}^{-\frac{1}{2}}} \right) \right)&\le \Psi \left( \exp \left( \frac{M{{\mathbf{1}}_{\mathcal{H}}}-{{A}^{-\frac{1}{2}}}B{{A}^{-\frac{1}{2}}}}{M-m}\ln f\left( m \right)+\frac{{{A}^{-\frac{1}{2}}}B{{A}^{-\frac{1}{2}}}-m{{\mathbf{1}}_{\mathcal{H}}}}{M-m}\ln f\left( M \right) \right) \right) \\ 
& \le \beta {{\mathbf{1}}_{\mathcal{K}}}+\alpha f\left( \Psi \left( {{A}^{-\frac{1}{2}}}B{{A}^{-\frac{1}{2}}} \right) \right),  
\end{aligned}
\end{equation}
where $\Psi $ is a normalized positive linear map from $\mathbb{B}\left( \mathcal{H} \right)$ to $\mathbb{B}\left( \mathcal{K} \right)$.

By taking $\Psi \left( X \right):=\Phi {{\left( A \right)}^{-\frac{1}{2}}}\Phi \left( {{A}^{\frac{1}{2}}}X{{A}^{\frac{1}{2}}} \right)\Phi {{\left( A \right)}^{-\frac{1}{2}}}$, where $\Phi $ is an arbitrary normalized positive linear map in \eqref{01}, we obtain the desired result \eqref{02}.
\end{proof}

\medskip

In the sequel, we use the notation $A{{\natural}_{v}}B={{A}^{\frac{1}{2}}}{{\left( {{A}^{-\frac{1}{2}}}B{{A}^{-\frac{1}{2}}} \right)}^{v}}{{A}^{\frac{1}{2}}}$, $\left( v\in \mathbb{R} \right)$. The following corollary follows by setting $f\left( t \right)={{t}^{p}}$ $\left( p\le 0 \right)$ in the previous theorem.
\begin{corollary}\label{004.1}
	Let $\Phi $ be a normalized positive linear map from $\mathbb{B}\left( \mathcal{H} \right)$ to $\mathbb{B}\left( \mathcal{K} \right)$ and $A,B\in \mathbb{B}\left( \mathcal{H} \right)$ be two strictly positive operators such that $mA\le B\le MA$ for some scalars $0<m<M$. Then for a given $\alpha \in \mathbb{R}$, 
	\[\begin{aligned}
	\Phi \left( A{{\natural}_{p}}B \right)&\le \Phi \left( {{A}^{\frac{1}{2}}}\exp \left( \frac{M{{\mathbf{1}}_{\mathcal{H}}}-{{A}^{-\frac{1}{2}}}B{{A}^{-\frac{1}{2}}}}{M-m}\ln {{m}^{p}}+\frac{{{A}^{-\frac{1}{2}}}B{{A}^{-\frac{1}{2}}}-m{{\mathbf{1}}_{\mathcal{H}}}}{M-m}\ln {{M}^{p}} \right){{A}^{\frac{1}{2}}} \right) \\ 
	& \le \beta \Phi \left( A \right)+\alpha \left( \Phi \left( A \right){{\natural}_{p}}\Phi \left( B \right) \right),
	\end{aligned}\]
	holds for $p\le 0$, where $\beta $ is defined by \eqref{be}.
\end{corollary}
Fujii and Seo \cite[Theorem 2.2]{02} showed the following operator inequality: Let $A,B\in \mathbb{B}\left( \mathcal{H} \right)$ be two positive operators and $\Phi $ be a normalized positive linear map, then 
	\begin{equation}\label{05}
	\Phi \left( A \right){{\natural}_{p}}\Phi \left( B \right)\le \Phi \left( A{{\natural}_{p}}B \right)\quad\text{ for }p\in \left[ -1,0 \right).
	\end{equation}

\medskip

The following corollary is a complementary result for \eqref{05}. The proof is immediate by using Corollary \ref{004.1}.
\begin{corollary}\label{04}
	Let $\Phi $ be a normalized positive linear map from $\mathbb{B}\left( \mathcal{H} \right)$ to $\mathbb{B}\left( \mathcal{K} \right)$ and $A,B\in \mathbb{B}\left( \mathcal{H} \right)$ be two strictly positive operators such that $mA\le B\le MA$ for some scalars $0<m<M$ and $p\le 0$. 
	\begin{itemize}
		\item[(i)] As a ratio type reverse of inequality \eqref{05} we have:
		\[\begin{aligned}
		\Phi \left( A{{\natural}_{p}}B \right)&\le \Phi \left( {{A}^{\frac{1}{2}}}\exp \left( \frac{M{{\mathbf{1}}_{\mathcal{H}}}-{{A}^{-\frac{1}{2}}}B{{A}^{-\frac{1}{2}}}}{M-m}\ln {{m}^{p}}+\frac{{{A}^{-\frac{1}{2}}}B{{A}^{-\frac{1}{2}}}-m{{\mathbf{1}}_{\mathcal{H}}}}{M-m}\ln {{M}^{p}} \right){{A}^{\frac{1}{2}}} \right) \\ 
		& \le K\left( m,M,p \right)\left( \Phi \left( A \right){{\natural}_{p}}\Phi \left( B \right) \right),  
		\end{aligned}\]
		where $K\left( m,M,p \right)$ is defined by \eqref{9}.
		\item[(ii)] As a difference type reverse of inequality \eqref{05} we have:
		\[\begin{aligned}
		\Phi \left( A{{\natural}_{p}}B \right)&\le \Phi \left( {{A}^{\frac{1}{2}}}\exp \left( \frac{M{{\mathbf{1}}_{\mathcal{H}}}-{{A}^{-\frac{1}{2}}}B{{A}^{-\frac{1}{2}}}}{M-m}\ln {{m}^{p}}+\frac{{{A}^{-\frac{1}{2}}}B{{A}^{-\frac{1}{2}}}-m{{\mathbf{1}}_{\mathcal{H}}}}{M-m}\ln {{M}^{p}} \right){{A}^{\frac{1}{2}}} \right) \\ 
		& \le C\left( m,M,p \right)\Phi \left( A \right)+\Phi \left( A \right){{\natural}_{p}}\Phi \left( B \right),
		\end{aligned}\]
		where $C\left( m,M,p \right)$ is defined by \eqref{cm}.
	\end{itemize}
\end{corollary}

We close this paper by presenting a result on the inequalities for the Tsallis relative operator entropy.
The Tsallis relative operator entropy with negative parameter  was introduced in \cite{04} as
\begin{equation}\label{09}
{{T}_{p}}\left( A|B \right)=\frac{A{{\natural}_{p}}B-A}{p}\quad\text{ for }p<0.
\end{equation}
Research in this field includes obtaining new inequalities and refining existing
ones.  For example, in \cite[Theorem 3.1 ($2'$)]{02}, the following inequality has been already shown:
\begin{equation}\label{06}
\Phi \left( {{T}_{p}}\left( A|B \right) \right)\le {{T}_{p}}\left( \Phi \left( A \right)|\Phi \left( B \right) \right)\quad\text{ for }p\in \left[ -1,0 \right).
\end{equation}

\medskip

We shall give complementary inequalities to the inequality \eqref{06}, thanks to Corollary \ref{04}. 
\begin{theorem}
	Let $\Phi $ be normalized positive linear map from $\mathbb{B}\left( \mathcal{H} \right)$ to $\mathbb{B}\left( \mathcal{K} \right)$ and $A,B\in \mathbb{B}\left( \mathcal{H} \right)$ be two strictly positive operators such that $mA\le B\le MA$ for some scalars $0<m<M$. Then for $p\in \left[ -1,0 \right)$,
	\begin{equation}\label{0010}
\begin{aligned}
& \Phi \left( {{T}_{p}}\left( A|B \right) \right) \\ 
&\quad \ge \frac{1}{p}\Phi \left( {{A}^{\frac{1}{2}}}\exp \left( \frac{M{{\mathbf{1}}_{\mathcal{H}}}-{{A}^{-\frac{1}{2}}}B{{A}^{-\frac{1}{2}}}}{M-m}\ln {{m}^{p}}+\frac{{{A}^{-\frac{1}{2}}}B{{A}^{-\frac{1}{2}}}-m{{\mathbf{1}}_{\mathcal{H}}}}{M-m}\ln {{M}^{p}} \right){{A}^{\frac{1}{2}}}-A \right) \\ 
&\quad \ge {{T}_{p}}\left( \Phi \left( A \right)|\Phi \left( B \right) \right)-\left( \frac{1-K\left( m,M,p \right)}{p} \right)\left( \Phi \left( A \right){{\natural}_{p}}\Phi \left( B \right) \right),
\end{aligned}	
	\end{equation}
			where $K\left( m,M,p \right)$ is defined by \eqref{9}
and
\begin{equation}\label{0020}
\begin{aligned}
& \Phi \left( {{T}_{p}}\left( A|B \right) \right) \\ 
&\quad \ge \frac{1}{p}\Phi \left( {{A}^{\frac{1}{2}}}\exp \left( \frac{M{{\mathbf{1}}_{\mathcal{H}}}-{{A}^{-\frac{1}{2}}}B{{A}^{-\frac{1}{2}}}}{M-m}\ln {{m}^{p}}+\frac{{{A}^{-\frac{1}{2}}}B{{A}^{-\frac{1}{2}}}-m{{\mathbf{1}}_{\mathcal{H}}}}{M-m}\ln {{M}^{p}} \right){{A}^{\frac{1}{2}}}-A \right) \\ 
&\quad \ge {{T}_{p}}\left( \Phi \left( A \right)|\Phi \left( B \right) \right)+\frac{C\left( m,M,p \right)}{p}\Phi \left( A \right), 
\end{aligned}
\end{equation}
where $C\left( m,M,p \right)$ is defined by \eqref{cm}.
\end{theorem}
\begin{proof}
It follows from Corollary \ref{04} (i) that
\begin{equation}\label{010}
\begin{aligned}
& \Phi \left( \frac{A{{\natural}_{p}}B-A}{p} \right) \\ 
&\quad \ge \frac{1}{p}\Phi \left( {{A}^{\frac{1}{2}}}\exp \left( \frac{M{{\mathbf{1}}_{\mathcal{H}}}-{{A}^{-\frac{1}{2}}}B{{A}^{-\frac{1}{2}}}}{M-m}\ln {{m}^{p}}+\frac{{{A}^{-\frac{1}{2}}}B{{A}^{-\frac{1}{2}}}-m{{\mathbf{1}}_{\mathcal{H}}}}{M-m}\ln {{M}^{p}} \right){{A}^{\frac{1}{2}}}-A \right) \\ 
&\quad \ge \frac{K\left( m,M,p \right)\left( \Phi \left( A \right){{\natural}_{p}}\Phi \left( B \right) \right)-\Phi \left( A \right)}{p}. \\ 
\end{aligned}
\end{equation}
On account of \eqref{09} the inequalities in \eqref{010} are equivalent to \eqref{0010}. From Corollary \ref{04} (ii) we have
\[\begin{aligned}
& \Phi \left( \frac{A{{\natural}_{p}}B-A}{p} \right) \\ 
&\quad \ge \frac{1}{p}\Phi \left( {{A}^{\frac{1}{2}}}\exp \left( \frac{M{{\mathbf{1}}_{\mathcal{H}}}-{{A}^{-\frac{1}{2}}}B{{A}^{-\frac{1}{2}}}}{M-m}\ln {{m}^{p}}+\frac{{{A}^{-\frac{1}{2}}}B{{A}^{-\frac{1}{2}}}-m{{\mathbf{1}}_{\mathcal{H}}}}{M-m}\ln {{M}^{p}} \right){{A}^{\frac{1}{2}}}-A \right) \\ 
&\quad \ge \frac{C\left( m,M,p \right)\Phi \left( A \right)+\Phi \left( A \right){{\natural}_{p}}\Phi \left( B \right)-\Phi \left( A \right)}{p}, \\ 
\end{aligned}\]
which is equivalent to \eqref{0020}.
\end{proof}
\vskip 0.4 true cm
\noindent{\bf Acknowledgement.} We would like to express our hearty thanks to Prof. Jadranka Mi\'ci\'c for her kind advice.
The author (S.F.) was partially supported by JSPS KAKENHI Grant Number
16K05257. 
\bibliographystyle{alpha}

\vskip 0.4 true cm

\tiny(S. Furuichi) Department of Information Science, College of Humanities and Sciences, Nihon University, 3-25-40, Sakurajyousui, Setagaya-ku, Tokyo, 156-8550, Japan.

{\it E-mail address:} furuichi@chs.nihon-u.ac.jp

\vskip 0.4 true cm

\tiny(H.R. Moradi) Department of Mathematics, Payame Noor University (PNU), P.O. Box 19395-4697, Tehran, Iran.

{\it E-mail address:} hrmoradi@mshdiau.ac.ir

\end{document}